\documentclass[reqno,10pt]{amsart}
\usepackage{latexsym,amsmath,amsthm,amssymb}%,fullpage}%,bbm}
\date{}
\newtheorem{theorem}{Theorem}[section]
\newtheorem{lemma}[theorem]{Lemma}
\newtheorem{corollary}[theorem]{Corollary}%[section]
\newtheorem{proposition}[theorem]{Proposition}%[section]
\theoremstyle{remark}
\newtheorem{example}[theorem]{Example}
\newtheorem{remark}[theorem]{Remark}%[section]

\theoremstyle{definition}

\newtheorem{problem}[theorem]{Problem}

\newcommand{\dR}{\ensuremath{\mathbb{R}}} %reels
\newcommand{\R}{\dR}

\begin{document}

\title{Mass transportation and contractions }

\author{Alexander V. Kolesnikov}

\begin{abstract}
According to a celebrated result of L. Caffarelli, every optimal mass transportation mapping 
pushing forward the standard Gaussian measure onto 
a  log-concave measure $e^{-W} dx$ with $D^2 W \ge \mbox{Id}$  is  $1$-Lipschitz.
We present  a short survey of related results and various applications.
\end{abstract}

\maketitle


\section{Introduction}

Given a positive number $\alpha$ we say that a mapping $T : \R^d \to \R^d$ is  $\alpha$-Lipschitz if 
$$
|T(x)-T(y)| \le \alpha |x-y|.
$$
For a smooth $T$ this is equivalent to the following: 
$$\sup_{x \in \R^d}\| DT (x)\| \le \alpha,$$ where $\| \cdot \|$ is the operator norm.
For the case $\alpha=1$ we say that $T$ is a contraction.

Similarly, a mapping  $T : X \to Y$  between metric spaces  is called  contraction, if
$\rho_Y(T(x_1), T(x_2)) \le \rho_X(x_1, x_2)$.

Let $\mu$ be a Borel measure on  a metric space  $(M, \rho)$.
Given a Borel set $A \subset M$  we  define the corresponding boundary measure $\mu^+$ of   $\partial A$ 
$$
\mu^{+}(\partial A) = \underline{\lim}_{h \to 0} \frac{\mu(A_h) - \mu(A)}{h},
$$
where $A_h = \{x: \rho(x,A) \le h\}$.

A set $A$ is called isoperimetric if it has the minimal surface measure among of all the sets 
with the same  measure $\mu(A)$. The isoperimetric profile
$\mathcal{I}_{\mu}$ of $\mu$  is defined as the following function
$$
\mathcal{I}_{\mu}(t) = \inf \{ \mu^+(\partial A): \ \mu(A)=t\}.
$$

Generally, isoperimetric sets are not possible to find.
Nevertheless, bounds for isoperimetric functions (the so-called isoperimetric inequalities)  have many
applications in analysis, geometry and probability theory.
It is well-known, for instance, that isoperimetric inequalities imply Sobolev-type inequalities. 
See more in  
\cite{Gromov}, \cite{MilSch}, \cite{Ledoux}, \cite{Ros},  \cite{Vill}.

 Numerous applications of contractions in analysis, probability and geometry rely on the following fact:

{\it 
Let $X$, $Y$ be two metric spaces and $X$ is equipped with a measure $\mu$.
Assume that there exists a contraction $T : X \to Y$  between metric spaces $X$ and $Y$. 
Then the image measure $\nu = \mu \circ  T^{-1}$ has a better isoperimetric profile
}
$$
\mathcal{I}_{\nu} \ge \mathcal{I}_{\mu}.
$$

In this paper we study mainly a special case of optimal transportations of measures.
Given two Borel probability measures $\mu$ and $\nu$
we consider the optimal transportation map 
$T$ 
minimizing the cost $$W^2_2(\mu,\nu) = \int |x - T(x)|^2 \ d \mu$$ among of all the  maps 
pushing forward $\mu$ to $\nu$. The latter means that $\mu \circ T^{-1}(A) = \nu(A)$
for every Borel $A$.

If $\mu = \rho_0 \ dx$ and $\nu = \rho_1 \ dx$ are absolutely continuous, then
$T$ does exist and can be obtained from the solution to the corresponding Monge-Kantorovich transportation problem. Moreover,
this map is $\mu$-unique  and has the form $T =
\nabla \Phi$, where $\Phi$ is convex (see \cite{Vill}). Assuming smoothness of   $\Phi$, one can easily verify that
$\Phi$  solves the
following nonlinear PDE (the  Monge--Amp{\`e}re  equation):
$$
{\rho_1(\nabla \Phi)} \det D^2 \Phi = {\rho_0}.
$$

This paper contains an overview of the results related to the contractivity of optimal transportation mappings.
The first result in this direction has been established by L.~Caffarelli (see \cite{Caf}).
Let $\mu$ be the standard Gaussian measure 
$\mu = \frac{1}{(2\pi)^{d/2}} e^{-\frac{x^2}{2}} \ dx$ and $\nu = e^{-W} \ dx$ 
with $D^2 W \ge \mbox{Id}$, then the corresponding $T$ is a contraction.
This observation implies immediately the Bakry-Ledoux comparison theorem and various 
functional inequalities, including the log-Sobolev inequality for uniformly log-concave measures.
Among of other applications let us mention the Gaussian correlation conjecture and the Brascamp-Lieb
inequality. We  discuss several extensions of this result and some open problems.

\section{Caffarelli's contraction theorem}

\begin{remark}
The Theorem \ref{contr1} and Theorem \ref{contr2}  below will be both referred to as "Caffarelli's contraction theorem".
Note, however, that  the original formulation is given in Theorem \ref{contr2}.
\end{remark}

\begin{theorem}
\label{contr1}
{\bf (L. Caffarelli) }
Let $T=\nabla \Phi$ be the optimal transportation mapping pushing forward a  probability measure $\mu = e^{-V} dx$
onto a probability measure $\nu = e^{-W} dx$. Assume that $V$ and $W$ are twice continuously differentiable and $D^2 W \ge K$. Then for every unit vector $e$
$$
\sup_{x \in \R^d} \Phi^2_{ee}  \le \frac{1}{K} \sup_{x \in \R^d} V_{ee}.
$$
In particular, if $\mu$ is the standard Gaussian measure and $K \ge 1$, then $T$ is a contraction.
\end{theorem}
{\bf Sketch of the proof:}

{\bf 1) Maximum principle proof.}

The proof based on the maximum principle is formal but elegant. Functons $V, W$ and $\Phi$ are assumed to be sufficiently regular.
Note that smoothness of $\Phi$ can be justified  in some favorable situations ($V,W$ are smooth and satisfy certain growth assumptions, see Theorem 4.14 of \cite{Vill}).
By the change of variables formula
$$
e^{-V} = e^{-W(\nabla \Phi)} \det D^2 \Phi.
$$
Taking the logarithm of both sides we get
$$
V = W(\nabla \Phi) - \log \det D^2 \Phi.
$$
We fix some unit vector $e$ and differentiate this formula twice along $e$.
To this end we apply the following fundamental relation
$$
\partial_{e} \ln \det D^2 \Phi = \frac{\partial_{e} \det D^2 \Phi}{\det D^2 \Phi} = \mbox{\rm Tr} (D^2 \Phi)^{-1} D^2 \Phi_{e}.
$$
Differentiating this formula along another direction $v$   and using that
$$
D^2 \Phi_{v} (D^2\Phi)^{-1} + D^2 \Phi  \bigl[ (D^2\Phi)^{-1} \bigr]_{v}=0
$$
we obtain
$$
\partial_{e v} \ln \det D^2 \Phi = \mbox{\rm Tr} (D^2 \Phi)^{-1} D^2 \Phi_{e v}
-
\mbox{Tr} \Bigl[ (D^2 \Phi)^{-1} D^2 \Phi_{e} (D^2 \Phi)^{-1} D^2 \Phi_{v} \Bigr].
$$
Coming back to the change of variables formula we get
$$
V_e = \langle \nabla W(\nabla \Phi),  D^2 \Phi \cdot e \rangle - \mbox{\rm Tr} (D^2 \Phi)^{-1} D^2 \Phi_{e}
$$
and
\begin{align*}
V_{ee}
 & =
\langle D^2 W(\nabla \Phi)  D^2 \Phi \cdot e,  D^2 \Phi \cdot e \rangle
+ 
\langle \nabla W(\nabla \Phi),  \nabla \Phi_{ee} \rangle
\\&
-
\mbox{\rm Tr} (D^2 \Phi)^{-1} D^2 \Phi_{ee}
+
\mbox{Tr} \Bigl[ (D^2 \Phi)^{-1} D^2 \Phi_{e}\Bigr]^2.
\end{align*}

Now assume that   $\Phi_{ee}$ attains its maximum at $x_0$. Then 
$$
\nabla \Phi_{ee}(x_0)=0, \ D^2 \Phi_{ee} \le 0.
$$
Note that $\mbox{Tr} \bigl[ (D^2 \Phi)^{-1} D^2 \Phi_{e}\bigr]^2 >0 $ because it equals to
 $\mbox{Tr} C^2 $, where $$C =(D^2 \Phi)^{-1/2} D^2 \Phi_{e}  (D^2 \Phi)^{-1/2} $$ is a symmetric matrix.

Clearly, $\mbox{\rm Tr} (D^2 \Phi(x_0))^{-1} D^2 \Phi_{ee}(x_0) \le 0$ and one gets
$$
V_{ee}(x_0) \ge K \| D^2 \Phi(x_0) \cdot e\| \ge K \Phi^2_{ee}(x_0).
$$
Hence
$$
\sup_{x \in \R^d} \Phi^2_{ee} \le \frac{1}{K} \sup_{x \in \R^d} V_{ee}(x_0).
$$

{\bf 2) Incremental quotients proof}

Instead of differentiating the Monge-Amp{\`e}re equation we consider 
the incremental quotient
$$
\delta_2 \Phi(x) = \Phi(x+th) + \Phi(x-th)
-2\Phi(x) \ge 0
$$
for some fixed vector $h \in \R^d$ with $|h|=1$.
By approximation, one can assume that  $\mbox{\rm supp}(\nu)$ is a bounded convex domain and $V$, $W$ are locally H{\"o}lder.
Caffarelli's regularity theory assures that $\Phi \in C^{2, \alpha}_{loc}(\R^d)$. 

In addition, again by approximation, one can assume that $\mu$ has at most Gaussian decay, meaning that
$V(x) \le C_1 + C_2 |x|^2$ for some $C_1, C_2 \ge 0$.
Then  the following lemma holds  (see Lemma 4 in \cite{Caf})
\begin{lemma}
$\lim_{x\to \infty} \delta_2 \Phi(x) = 0$.
\end{lemma}

Thus there exists a maximum point $x_0$ of $\delta_2 \Phi(x)$.
Differentiating at $x_0$ yields
\begin{equation}
\label{maxim1}
\nabla \Phi(x_0+th) + \nabla \Phi(x_0-th) =2 \nabla \Phi(x_0),
\end{equation}
$$
D^2 \Phi(x_0+th) + D^2 \Phi(x_0-th) \le 2 D^2 \Phi(x_0).
$$
It follows from the concavity of the determinant that
\begin{align*}
\det  D^2 \Phi(x_0)
&
\ge
\det \Bigl( \frac{D^2 \Phi(x_0+th) + D^2 \Phi(x_0-th)}{2} \Bigr)
\\&
\ge
\Bigl( \det D^2 \Phi(x_0+th) \ \det D^2 \Phi(x_0-th)\Bigr)^{\frac{1}{2}}.
\end{align*}
Applying the change of variables formula
$
\det D^2 \Phi = e^{W(\nabla \Phi) -V}
$ one finally gets
\begin{align}
\label{V-W}
V(x_0 + th) + V(x_0 - th) & - 2 V(x_0) \ge
\\&
\nonumber
W(\nabla \Phi(x_0 + th)) + W(\nabla \Phi(x_0 -th))
- 2 W(\nabla \Phi(x_0)).
\end{align}

It follows from  (\ref{maxim1}) that
$ v:=
\nabla \Phi(x_0+th) - \nabla \Phi(x_0) = \nabla \Phi(x_0) -  \nabla \Phi(x_0-th).
$
Hence we get by (\ref{V-W}) that
$$
\sup V_{hh} \cdot  t^{2} \ge  K |\nabla \Phi(x_0+th) - \nabla \Phi(x_0)|^{2}
=  K | \nabla \Phi(x_0-th) - \nabla \Phi(x_0)|^{2}
= K |v|^{2}.
$$
By convexity of $\Phi$
\begin{align*}
\Phi(x_0+th) + \Phi(x_0-th) - 2\Phi(x_0)
&
\le t \langle \nabla \Phi(x_0+th) - \nabla \Phi(x_0-th), h \rangle
\\&
= 2t \langle v,h \rangle \le 2t |v|.
\end{align*}
Finally
$$
\frac{\sup_{x \in \R^d} V_{hh}}{K} 
\ge \Bigl( \frac{\delta_2 \Phi}{2t^2}\Bigr)^{2}.
$$
This clearly implies
$$
\Phi_{hh} \le 2 C
$$
with $C = \sqrt{\frac{\sup_{x \in \R^d} V_{hh}}{K}}.$
But this estimate is worse that the desired one.
To get the sharp estimate we repeat the arguments  and use 
the additional information that $\Phi_{hh} \le a_0 C$,
where $a_0 = 2$. Apply the identity
$$
\Phi(x_0+th) + \Phi(x_0-th) - 2\Phi(x_0)  = \int_{0}^{t} \langle \nabla \Phi(x_0 + sh) - \nabla \Phi(x_0 -s h), h \rangle \  ds.
$$
By convexity of $\Phi$
$\langle \nabla \Phi(x_0 + sh) - \nabla \Phi(x_0 -s  h), h \rangle \le \langle \nabla \Phi(x_0 + th) - \nabla \Phi(x_0 - t h), h \rangle$.
 One has
$$
\Phi(x_0+th) + \Phi(x_0-th) - 2\Phi(x_0)  
\le 
\int_0^t  \min \bigl( 2  a_0 C s, 2 |v| \bigr) \ ds.
$$
Computing the right-hand side and taking into account that $|v| \le Ct$, we get that
$$
\Phi(x_0+th) + \Phi(x_0-th) - 2\Phi(x_0)  
\le 
a_1 Ct^2.
$$
where $a_1 = \frac{3}{2}$.
Hence $\Phi_{hh} \le a_1 C$.
Repeating this arguments infinitely many times we get that $\Phi_{hh} \le a_n C$ and $\lim_n a_n =1$. The proof is complete.

{\bf 3) Proof via $L^p$-estimates}

See Section 6.

\begin{remark}
We note  that the original result from \cite{Caf} was slightly different from the result stated above.
Here is the exact statement proved by Caffarelli.
\end{remark}

\begin{theorem}
\label{contr2} {\bf (L. Caffarelli)}
Let $\mu = e^{-Q} \ dx$ be any Gaussian measure. Then for any measure  $\nu = e^{-Q - P} \ dx$,
where $P$ is convex, the corresponding optimal transportation $T$ is a contraction.
\end{theorem}
{\bf Sketch of the proof:}
Let us apply the maximum principle arguments. We are looking for a maximum of $\Phi_{ee}(x)$
among of all unit $e$ and $x \in \R^d$. Apply the relation obtained above
\begin{align*}
Q_{ee}
& =
\langle D^2 (Q+P)(\nabla \Phi)  D^2 \Phi \cdot e,  D^2 \Phi \cdot e \rangle
+ 
\langle \nabla (Q+P)(\nabla \Phi),  \nabla \Phi_{ee} \rangle
\\&
-
\mbox{\rm Tr} (D^2 \Phi)^{-1} D^2 \Phi_{ee}
+
\mbox{Tr} \Bigl[ (D^2 \Phi)^{-1} D^2 \Phi_{e}\Bigr]^2.
\end{align*}
By the same reasons as above
\begin{align*}
Q_{ee}
\ge
\langle D^2 (Q+P)(\nabla \Phi)  D^2 \Phi \cdot e,  D^2 \Phi \cdot e \rangle.
\end{align*}
Now take into account that $P$ is convex and, in addition, $e$ must be an eigenvector of $D^2 \Phi$. Hence we obtain
$$
Q_{ee} \ge \Phi^2_{ee} \cdot Q_{ee}(\nabla \Phi).
$$
Taking into account that $Q_{ee}$ is constant, we obtain the claim.

\section{General (uniformly) log-concave measures}

The incremental quotients proof  can be easily extended to the case of
measures which are uniformly log-concave in a generalized case. 
The latter means that the potential $W$ satisfies
$$
W(x+y)+W(x-y) - W(x) \ge \delta(|y|)
$$
for some increasing function $\delta$.
The following result has been proved in \cite{Kol}.

\begin{theorem}
\label{hoelder}
Assume that $V$ and $W$ satisfy
$$
V(x+ y) + V(x-y) -2V(x) \le A_p |y|^{p+1},
$$
$$
W(x+ y) + W(x-y) -2W(x) \ge A_q |y|^{q+1},
$$
for some $0 \le p \le 1$, $1 \le q$, $A_p>0$, $A_q>0$.

Then  $\Phi$
satisfies
\begin{equation}
\label{var-Hoelder}
\Phi(x+th) + \Phi(x-th)
-2\Phi(x) \le 2\Bigl( \frac{A_p}{A_q}\Bigr)^{\frac{1}{q+1}} t^{1+\alpha}
\end{equation}
for every unit vector $h \in \R^d$
with $\alpha = \frac{p+1}{q+1}$.
\end{theorem}

\begin{remark}
The constant  in (\ref{var-Hoelder}) is not optimal in general.
\end{remark}

It follows from (\ref{var-Hoelder}) that $\nabla \Phi$ is globally H{\"o}lder.
This fact is actualy true without any convexity assumption  
on $\Phi$, but the convex case is more simple and the result follows from the following lemma communicated to the authors
by Sasha Sodin.

\begin{lemma}
\label{Sodin-lem}
For every convex $f$ and unit vector $h$  one has
$$
|\nabla f(x+th) -f(x)|
\le \frac{2}{t} \sup_{v: |v|=1}  \Bigl( f(x+2tv) + f(x-2tv) - 2 f(x)
 \Bigr).
$$
\end{lemma}

Using this lemma one can extend the H{\"o}lder regularity  result.
\begin{theorem}
\label{MS-conc}
Assume that
$$
V(x+ y) + V(x-y) -2V(x) \le |y|^{2},
$$
and
$$
W(x+y)+W(x-y) - W(x) \ge \delta(|y|)
$$
with some non-negative increasing function $\delta$.
Then
$$
|\nabla \Phi(x) - \nabla \Phi(y) | \le 8 \delta^{-1}(4|x-y|^2).
$$
\end{theorem}

Applying this estimate one can transfer the famous
Gaussian Sudakov-Tsirelson isoperimetric inequality to any (generalized) uniform log-concave measure.
Recall (see \cite{Bo}), that the standard Gaussian measure
$\gamma$ satisfies the Gaussian isoperimetric inequality
$$
\gamma(A^r) \ge \Phi(\Phi^{-1}(\gamma(A)) +r),
$$
where $A^r = \{x\in \R^d: \exists a \in A: |a-x|<r\}$, $\Phi(x) = \frac{1}{\sqrt{2\pi}} \int_{-\infty}^{x} e^{-\frac{t^2}{2}} dt$.

Consequently,  applying  Theorem \ref{MS-conc} to $\mu=\gamma$ and 
 $\nu = e^{-W} dx$ with $W$
 satisfying
$$
W(x+y)+W(x-y) - W(x) \ge \delta(|y|),
$$
we get
$$
\nu\bigl( A_{r} \bigr) \ge  \Phi\Bigl(\Phi^{-1}(\nu(A)) +  \frac{1}{2} \sqrt{\delta(r/8)} \Bigr).
$$
In particular, $\nu$ admits the following dimension-free concentration
property:
$$
\nu\bigl( A_{r}  \bigr) \ge  1- \frac{1}{2} \exp \Bigl( -\frac{1}{8} \ \delta(r/8) \Bigr)
$$
with $\nu(A) \ge 1/2$. A similar result has been established  by S.~Sodin and E.~Milman in \cite{MilSod} by localization arguments.
Note that according to results of E.~Milman \cite{Milman08}  concentration and isoperimetric inequalities are in a sence equivalent
for log-concave measures.

\section{Lebesgue measure on a convex set}

In this section we discuss the following problem.

\begin{problem}
Given a nice (product) probability  measure $\mu$ (e.g. Gaussian or exponential)
estimate effectively the Lipschitz constant of the optimal mapping pushing forward
$\mu$ onto the normalized Lebesgue measure on a convex set $K$. 
\end{problem}

This problem was motivated in particular by the famous Kannan-Lov{\'a}sz-Simonovits conjecture (KLS-conjecture).
Recall that the Cheeger $C_{ch}(K)$ constant of a convex body $K$ is the smallest constant such that the inequality
$$
\int_K \Bigl|f - \frac{1}{\lambda(K)} \int_K f dx \Bigr | \ dx \le C_{ch}(K) \int_{K} |\nabla f| \ dx
$$ holds
for every smooth $f$.

{\bf KLS conjecture}.
There exists an universal constant $c$ such that
$$
C_{ch}(K) \le c
$$
for every convex  $K \subset \R^d$ satisfying
$$
 \int_K  x_i \ dx = 0, \ \  \frac{1}{\lambda(K)} \int_K  x_i x_j \ dx = \delta_{i}^{j}.
$$

More on the KLS conjecture see in \cite{KLS}, \cite{Bob07}, \cite{Milman08}.

Some estimates of the Lipschitz constant for optimal transportation of convex bodies have been obtained in \cite{Kol}. The arguments below generalize 
the maximum principle proof of Caffarelli. Let $\nabla \Phi$ be the optimal transportation mapping pushing forward $e^{-V} dx$ to $\frac{1}{\lambda(K)} \lambda|_{K}$.
Let us fix a unit vector $h$ . We are looking for a function
$\psi$ such that 
$$
\psi(\Phi_h) + \log \Phi_{hh}
$$
is bounded from above.
Assume that $x_0$ is the maximum point. One has at this point
\begin{equation}
\label{max-grad}
\psi'(\Phi_h) \nabla \Phi_h + \frac{1}{\Phi_{hh}} \nabla \Phi_{hh}=0
\end{equation}
\begin{equation}
\label{max-hess}
\psi''(\Phi_h) \nabla \Phi_h \oplus \nabla \Phi_h  + \psi'(\Phi_h) D^2 \Phi_h + \frac{1}{\Phi_{hh}} D^2 \Phi_{hh} - \frac{1}{\Phi^2_{hh}} 
 \nabla \Phi_{hh} \oplus \nabla \Phi_{hh} \le 0.
\end{equation}
Differentiation the change of variables formula gives (see Section 1)
$$
V_h = - \mbox{\rm Tr} (D^2 \Phi)^{-1} D^2 \Phi_{h},
$$
$$
V_{hh}
=
-
\mbox{\rm Tr} (D^2 \Phi)^{-1} D^2 \Phi_{hh}
+
\mbox{Tr} \Bigl[ (D^2 \Phi)^{-1} D^2 \Phi_{h}\Bigr]^2.
$$
Multiply (\ref{max-hess}) by $(D^2 \Phi)^{-1}$, take the trace and plug in the expression for 
$V_{hh}$ into the formula. One  obtains
\begin{align*}
V_{hh} &
\ge - \frac{1}{\Phi_{hh}} \mbox{\rm Tr} \bigl[ (D^2 \Phi)^{-1} \cdot \nabla \Phi_{hh} \oplus \nabla \Phi_{hh} \bigr]
+ {\Phi_{hh}} \cdot \psi^{''}(\Phi_h) \mbox{\rm Tr} \bigl[ (D^2 \Phi)^{-1} \cdot \nabla \Phi_{h} \oplus \nabla \Phi_{h} \bigr]
\\&
+
\Phi_{hh} \cdot \psi'(\Phi_h)  \mbox{\rm Tr} \Bigl[ (D^2 \Phi)^{-1} D^2 \Phi_{h}\Bigr]
+ \mbox{Tr} \Bigl[ (D^2 \Phi)^{-1} D^2 \Phi_{h}\Bigr]^2.
\end{align*}
Remark that $ \mbox{\rm Tr} \bigl[ (D^2 \Phi)^{-1} \cdot \nabla \Phi_{h} \oplus \nabla \Phi_{h} \bigr] =\Phi_{hh}$. 
One obtains  from (\ref{max-grad}) that
$ \nabla \Phi_{hh} =- \Phi_{hh} \cdot \psi'(\Phi_h) \nabla \Phi_h$. Plugging this into the inequality for $V_{hh}$ one gets
$$
V_{hh} 
\ge \Phi^2_{hh} \bigl[ \psi^{''} - (\psi')^2\bigr] \circ \Phi_h
+
\Phi_{hh} \cdot \psi'(\Phi_h)  \mbox{\rm Tr} \Bigl[ (D^2 \Phi)^{-1} D^2 \Phi_{h}\Bigr]
+ \mbox{Tr} \Bigl[ (D^2 \Phi)^{-1} D^2 \Phi_{h}\Bigr]^2.
$$
Note that 
$$
 \mbox{\rm Tr} \Bigl[ (D^2 \Phi)^{-1} D^2 \Phi_{h}\Bigr] =  \mbox{\rm Tr} C, \ \ 
 \mbox{Tr} \Bigl[ (D^2 \Phi)^{-1} D^2 \Phi_{h}\Bigr]^2 =  \mbox{Tr} C^2,
$$
where
$$
 C = (D^2 \Phi)^{-1/2} (D^2 \Phi_{h}) (D^2 \Phi)^{-1/2}
$$
is a symmetric matrix.
Hence, by the Cauchy inequality
$$
V_{hh} 
\ge \Phi^2_{hh} \Bigl[ \psi^{''} - \Bigl(1 + \frac{d}{4}\Bigr) (\psi')^2\Bigr] \circ \Phi_h.
$$
Now assume that $V_{hh}$ is bounded from above by a constant $C$. Let $\psi$ be a function satisfying
$$\psi^{''} - \Bigl(1 + \frac{d}{4}\Bigr) (\psi')^2  \ge e^{2 \psi}.$$  Then we get
$$
C \ge \Phi^2_{hh}(x_0) e^{2 \psi(\Phi_h(x_0))} = \sup_{x \in \R^d} \Phi^2_{hh} e^{2 \psi(\Phi_h)}.
$$

In particular, choosing carefully $\psi$ one can obtain the following statement (see \cite{Kol} for details).

\begin{theorem}
\label{set-image}
1) Optimal transportation $T$ of the standard Gaussian measure $\gamma$ onto $\frac{1}{\lambda(K)} \lambda|_{K}$, where $K$ is convex,
satisfies
$$
\|DT\| \le  c \sqrt{d} \ \mbox{\rm diam}(K),
$$
where $c$ is an universal constant and  $\mbox{\rm diam}(K)$ is the diameter of $K$.

2) Optimal transportation $T$ between $\mu = e^{-V} \ dx$ and  $\frac{1}{\lambda(K)} \lambda|_{K}$, with
$ V_{hh} \le C $, $|V_h| \le C$ for some $C$,  satisfies
$$
\|DT\| \le  c \ \mbox{\rm diam}(K),
$$
where $c$  depends only on $C$.
\end{theorem}

Unfortunately, estimates of Theorem \ref{set-image}
are not strong enough to recover even known results on the Cheeger  constant for convex bodies. 
This gives raise to the following problem.

\begin{problem}
Does there exist any dimension-free estimate 
for $\|DT\|$, when $\mu = \gamma$ and $\nu = \frac{1}{\lambda(K)} \lambda|_{K}$?
The same for the case when $\mu$  is the product of exponentional distributions.
\end{problem}

Note, that it would be enough for our purpose to have a integral norm estimate $\int \|DT\|^p \ d\gamma$, $p \ge 1$.
This follows form the result of E. Milman \cite{Milman08} about equivalence of norms for log-concave measures.

\section{Contraction for the mass transport generated by semigroups}

A contraction result for another type of mass transport has been obtained recently in \cite{KimMilman} by Y.-H~Kim and E.~Milman.
The idea of the construction of this transportation  mapping goes back to J.~Moser.

Consider the diffusion semigroup $P_t = e^{tL}$ denerated by 
$$
L = \Delta- \langle \nabla V, \nabla \rangle = e^{V} \mbox{div} (e^{-V} \cdot \nabla)
$$
and the flow of probability measures
$$
\nu_t =P_t  (e^{-W+V}) \cdot \mu.
$$
Clearly, $\mu$ is  the invariant measure for $P_t$, $\nu_0 = \nu$, and $\nu_{\infty} = \mu$.

Let us write the transport equation  for $\nu_t$:
$$
\frac{d}{dt} \nu_t
= L P_t (e^{-W+V})  \cdot \mu
=  \mbox{div}  \bigl[ \nabla P_t (e^{-W+V}) \cdot e^{-V}  \bigr]
=  \mbox{div}  \bigl[ \nabla \log  P_t (e^{-W+V}) \cdot \nu_t  \bigr].
$$
The corresponding flow of diffeomorphisms
is governed by the equation
\begin{equation}
\label{Lagr}
\frac{d}{dt} S_t =  -\nabla \log  P_t (e^{-W+V}) \circ S_t, \  \  S_0 = \mbox{Id},
\end{equation}
where $\nu_t$ and $S_t$ are related by 
$$
{\nu_t} = \nu \circ S^{-1}_t.
$$
In particular, the limiting map $S_{\infty} = \lim_{t \to \infty} S_t$
pushes forward $\nu$ to $\mu$. We denote the inverse mappings by $T_t$:
$$
T_t \circ S_t = \mbox{Id}, \ \  T = \lim_{t \to \infty} T_t.
$$

The contraction property for $T=S^{-1}$ is equivalent to the expansion property of $S$. It is sufficient to show that
$(DS_t)^* DS_t \ge \mbox{Id}$.
Using  (\ref{Lagr}) one gets
$$
\frac{d}{dt} DS_t(x) = -DW_t(S_t) \cdot DS_t, \ \ W_t =   \nabla \log  P_t (e^{-W+V}).
$$
Hence
$$
\frac{d}{dt} (DS_t)^* DS_t  =
2 (DS_t)^* \cdot   DW_t(S_t) \cdot DS_t.
$$
Clearly, if 
$$
  DW_t(S_t)  = -D^2  \log  P_t (e^{-W+V}) \ge 0, 
$$
then $S_t$ has the desired expansion property.

Assume now that the function $U$ defined by
$$
\nu = e^{-U} \cdot \mu, \ U = W-V,
$$
 is convex. Then the property $- D^2  \log  P_t (e^{-W+V})  = - D^2 \log P_t e^{-U} \ge 0$ means that
$P_t$  {\it preserves log-concave functions}.
Thus we obtain

\begin{theorem}
Assume that $U$ is convex.
If $U_t = -\log P_t e^{-U}$ is a convex function for every $t \ge 0$, then every $T_t$ is a $1$-contraction.  

\end{theorem}

It should be noted that by a resulf from \cite{Kol2001} the property to preserve {\it all}
log-concave fuinctions do admit only diffusion semigroups with Gaussian kernels.
Nevertheless, Kim and Milman were able to show under certain symmetry assumptions
log-concavity is preserved. The proof is based on the application of the maximum principle.

They get, in particularly, the
following result (see \cite{KimMilman} for a more general statement).

\begin{theorem}
Assume that $\mu$ is a product mesure, $V$ and $U$ are convex functions, $U$ is unconditional 
$U(x_1, \cdots, x_n) = U(\pm x_1, \cdots, \pm x_n)$,
and $V(x) = \sum_{i=1}^d \rho_i(|x_i|)$
with $\rho_i''' \le 0$.

Then $T$ is a contraction.  In addition, the optimal transportation mapping $T_{opt}$ pushing forward $\mu$ onto $\nu$ is a contraction too.
\end{theorem}

Let us very briefly explain the idea of the proof. Let $t_0$ be the first moment when the convexity of $U_t$ fails.
Assume that the minimum  of $\partial_{ee} U_{t_0}$ is attained at some point $x_0$ for some direction $e$.
Then $(d/dt - \Delta) \partial_{ee} U_{t} |_{t_0, x_0} \le 0$. In addition, $\nabla \partial_e U_t =0$ and $\nabla \partial_{ee} U_t =0$.
Using this one can show that
$$(d/dt - \Delta) \partial_{ee} U_{t}|_{t_0, x_0} 
=
- \langle \nabla U_t, \nabla V_{ee}\rangle|_{t_0, x_0}.
$$
At the time $t_0$ the function $U_t$ is still convex and it is easy to see that the right-hand side schould be non-negative.
This leads to a contradiction.

\section{$L^p$-contractions}

In this section we discuss an $L^p$-generalization of the Caffarelli's theorem (see \cite{Kol2010}). 
The proof below is obtained  with the help of the so-called 
above-tangent formalism (see \cite{Kol2010}). The huge advantage of this approach is that no a priori regularity
of the function $\Phi$ is required. 
See \cite{Kol2010} for details and relations to the transportation inequalities.

\begin{remark}
The estimates obtined in this section can be considered as global dimenion-free Sobolev a priori estimates for the optimal transportation problem.
In particular, they can be generalized fo  infinite-dimensional measures.
\end{remark}

\begin{theorem}
\label{lp-est}
 Assume that $D^2 W \ge K \cdot \mbox{\rm Id}$. Then for every unit $e$, $p \ge 1$, one has
$$
K \| \Phi^2_{ee} \|_{L^{p}(\mu)} \le \| (V_{ee})_{+} \|_{L^{p}(\mu)},
$$
$$
K \| \Phi^2_{ee} \|_{L^{p}(\mu)} \le \frac{p+1}{2}  \| V^2_{e} \|_{L^{p}(\mu)}.
$$
\end{theorem}
\begin{proof}
 Fix unit vector $e$. According to a result of McCann \cite{McCann2} the change of variables formula
$$
V(x) = W(\nabla \Phi(x)) - \log \det D^2_a \Phi
$$
holds $\mu$-almost everywhere 
Here $ D^2_a \Phi$ is the absolutely continuous part of the second distributional 
derivative $D^2 \Phi$ (Alexandrov derivative). 
One has
$$
V(x+ te) -  V(x)  = 
 W(\nabla \Phi(x+te))-  W(\nabla \Phi(x))  - \log \Bigl[ ({\det}_a D^2 \Phi(x))^{-1}  \cdot {\det}_a D^2\Phi(x+te)\Bigr].
$$
By the uniform convexity of $W$
\begin{align*}
V(x+ te) -  V(x)  & \ge 
\langle \nabla \Phi(x+te)- \nabla \Phi(x), \nabla W(\nabla \Phi(x))  \rangle \\& + \frac{K}{2} |\nabla \Phi(x+te)- \nabla \Phi(x)|^2
 - \log \Bigl[ ({\det}_a D^2 \Phi(x))^{-1}  \cdot {\det}_a D^2\Phi(x+te)\Bigr].
\end{align*}
Multiply this identity by $(\delta_{te} \Phi)^p $, where $p \ge 0$ and
$$
\delta_{te} \Phi =  \Phi(x+te) + \Phi(x-te) -2 \Phi(x)
$$
and integrate over $\mu$.
We apply the following simple lemma.
\begin{lemma}
Let $\varphi: A \to \R$, $\psi: B \to \R$ be convex functions on convex sets $A$, $B$. 
Assume that $\nabla \psi(B) \subset A$.
Then
$$
\mbox{\rm div} (\nabla \varphi \circ \nabla \psi)
\ge \mbox{\rm Tr}\bigl[ D^2_a \varphi (\nabla \psi) \cdot D^2_a \psi \bigr] \ dx \ge 0,
$$
where $\mbox{\rm div}$ is the distributional derivative.
\end{lemma}
Integrating by parts and applying this lemma 
we get
\begin{align*}
 \int \langle   \nabla \Phi(x+te)  & - \nabla \Phi(x), \nabla W(\nabla \Phi(x))  \rangle 
(\delta_{te} \Phi)^p 
\ d\mu \\& =
\int \langle  \nabla \Phi(x+te) \circ (\nabla \Psi)- x, \nabla W(x)  \rangle 
(\delta_{te} \Phi)^p  \circ (\nabla \Psi)
\ d\nu
\\&
\ge
\int \Bigl( \mbox{Tr} \bigl[ D^2 _a\Phi(x+te) \cdot (D^2_a \Phi)^{-1} \bigr] \circ  (\nabla \Psi)  - d \Bigr) (\delta_{te} \Phi)^p  \circ (\nabla \Psi) \ d\nu
\\&
+ 
p \int \Big\langle \nabla \Phi(x+te) \circ (\nabla \Psi)- x, (D^2 \Psi)   \nabla \delta_{te} \Phi  \circ (\nabla \Psi)  \Big\rangle (\delta_{te} \Phi)^{p-1}\circ (\nabla \Psi)  \ d\nu.
\end{align*}
We note that
$$\mbox{Tr} A - d - \log \det A \ge 0$$
for any $A$ of the type $A=BC$, where $B$ and $C$
are symmetric and positive.
Indeed, 
$$
\mbox{Tr} A - d - \log \det A
=\mbox{Tr} C^{1/2} B C^{1/2} - d - \log \det C^{1/2} B C^{1/2}
= \sum_i \lambda_i - 1 - \log  \lambda_i,
$$
where $\lambda_i$ are eigenvalues of $C^{1/2} B C^{1/2}$.

Consequently
\begin{align*}
\int \bigl( V(x+ te) - & V(x) \bigr) (\delta_{te} \Phi)^p d\mu 
\ge
\frac{K}{2} \int |\nabla \Phi(x+te)- \nabla \Phi(x)|^2 (\delta_{te} \Phi)^p \ d\mu \\&
+ p \int \Big\langle \nabla \Phi(x+te) - \nabla \Phi(x), (D^2 \Psi) \circ \nabla \Phi(x)   \nabla \delta_{te} \Phi  \Big\rangle (\delta_{te} \Phi)^{p-1}  \ d\mu.
\end{align*}
Applying the same inequality to $-te$ and taking the sum we get
\begin{align*}
\int  & \bigl( V(x+ te) + V(x-te) - 2V(x) \bigr) (\delta_{te} \Phi)^p d\mu 
\\&
\ge
\frac{K}{2} \int |\nabla \Phi(x+te)- \nabla \Phi(x)|^2 (\delta_{te} \Phi)^p \ d\mu  + 
\frac{K}{2} \int |\nabla \Phi(x-te)- \nabla \Phi(x)|^2 (\delta_{te} \Phi)^p \ d\mu \\&
+ p \int \Big\langle \nabla \delta_{te} \Phi , (D^2_a \Phi)^{-1}    \nabla \delta_{te} \Phi  \Big\rangle (\delta_{te} \Phi)^{p-1}  \ d\mu.
\end{align*}
Note that the last term is non-negative.
Dividing by $t^{2p}$ and passing to the limit we obtain
\begin{equation}
\label{lp-sd+}
\int V_{ee}  \Phi_{ee}^p  \ d \mu 
\ge K
\int \| D^2 \Phi \cdot e\|^2   \Phi_{ee}^p  \ d \mu
+ p \int \langle (D^2 \Phi)^{-1}  \nabla \Phi_{ee}, \nabla \Phi_{ee} \rangle  \Phi_{ee}^{p-1} \ d \mu.
\end{equation}
For the proof of the first part we note that
$$
\int V_{ee} \Phi^p_{ee} \ d\mu 
\ge 
K \int  \Phi_{ee}^{p+2}  \ d \mu.
$$
Applying
the H{\"o}lder inequality
one gets
$$
 \| (V_{ee})_{+}\|_{L^{(p+2)/2}(\mu)}  \| \Phi_{ee}^p\|_{L^{(p+2)/p}(\mu)}
\ge \int V_{ee} \Phi^p_{ee} \ d\mu .
$$
This readily implies the result.

To prove the second part we integrate by parts the left-hand side
\begin{align*}
\int V_{ee} \Phi^{p}_{ee} \  d \mu &
=
-p \int V_e \Phi_{eee} \Phi^{p-1}_{ee} \ d\mu +  \int V^2_e \Phi^{p}_{ee} \  d \mu
\\&
= - p \int \langle \nabla \Phi_{ee}, V_e \cdot e \rangle \Phi^{p-1}_{ee} d\mu +  \int V^2_e \Phi^{p}_{ee} \  d \mu.
\end{align*}
By the Cauchy inequality the latter does not exceed
$$
 p \int \langle (D^2 \Phi)^{-1}  \nabla \Phi_{ee}, \nabla \Phi_{ee} \rangle  \Phi_{ee}^{p-1} \ d \mu
+ \frac{p}{4} \int V^2_e \langle D^2 \Phi e, e \rangle  \Phi_{ee}^{p-1} \ d \mu +  \int V^2_e \Phi^{p}_{ee} \  d \mu.
$$
Inequality (\ref{lp-sd+}) implies
$$
\frac{p+4}{4} \int V^2_e   \Phi_{ee}^{p} \ d \mu \ge 
 K \int |\nabla \Phi_e|^2 \Phi_{ee}^p  \ d \mu \ge K \int  \Phi_{ee}^{p+2}  \ d \mu.
$$
The rest of the proof is the same as in the first part.

\end{proof}

\begin{corollary}
In the limit  $p\to\infty$ we obtain  the contraction theorem of Caffarelli
$$
K \| \Phi_{ee} \|^2_{L^{\infty}(\mu)} \le \| (V_{ee})_{+} \|_{L^{\infty}(\mu)}.$$
\end{corollary}

A more difficult estimate for the operator norm $\| \cdot \|$ has been also obtained in \cite{Kol2010}.

\begin{theorem}
 Assume that $D^2 W \ge K \cdot \mbox{\rm Id}$. Then for every $r \ge 1$ one has
$$
K 
\Bigl( \int  \| D^2 \Phi\|^{2r} \ d \mu \Bigr)^{\frac{1}{r}}
\le \Bigl(  \int  \| (D^2 V)_{+} \|^r  \ d \mu \Bigr)^{\frac{1}{r}}.
$$
\end{theorem}

\section{Contractions for infinite measures}

In this section we investigate contractions of infinite measures.
Let us stress that unlike the probability case we don't have a natural probabilistic normalization of the total volume.

We start with the following $1$-dimensional example 

\begin{example}
%{ \bf (F. Morgan)}
Let $d=1$ and  $\mu= \lambda|_{\R^+}$, $\nu = I_{[0,+\infty)} \rho \ dx$
and $\rho \ge 1$. Then the standard monotone transportation $T$ is a contraction.
\end{example}
\begin{proof}
Indeed, this follows immediately from the explicit representation of  $T$
$$
\int_{0}^{T} \rho \ dx = x.
$$
\end{proof}

Let us investigate what happens for $d=2$ if the image measure is rotationally symmetric.

\begin{example} {\bf (F. Morgan)}
Let $d=2$ and  $\mu= \lambda$, $\nu = \Psi(r) \ dx$.
A natural transport mapping has the form
$$
T(x) = \varphi(r) \cdot n, \ \ n = \frac{x}{r}
$$
Clearly
$$
\nu(T(B_r)) = 2 \pi \int_{0}^{\varphi(r)} \ s \Psi(s) \ dr= \pi r^2
= \mu(B_r) .
$$
Let us compute $DT$ in the frame $(n,v)$, where $v = \frac{(-x_2,x_1)}{r}$.
One has
$$
\partial_n T = \varphi' \cdot n  \  \ \
\partial_v T = \frac{\varphi}{r} \cdot v. 
$$
Clearly, a necessary and sufficient condition for $T$ to be a contraction is the following:
$$
\varphi' \le 1
$$
or $\psi' \ge 1$ for $\psi = \varphi^{-1}$.
From the change of variables formula we obtain
$$
\psi(r) = \sqrt{2 \int_0^r s \Psi(s)  \ ds}.
$$
Condition $\psi' \ge 1$ is equivalent to
$
\int_0^r s \Psi (s) \ ds \le \frac{(r \Psi(r))^2}{2}.
$
The latter holds, for instance, if 
$$
(s \Psi (s))' \ge 1.
$$
Indeed, in this case
$$
\int_0^r s \Psi (s) \ ds \le \int_0^r s \Psi (s) (s \Psi (s))' \ ds =\frac{(r \Psi(r))^2}{2}.
$$
\end{example}

\begin{example}
Similarly in dimension $d$, a sufficient condition for the 
transportation mapping $T = \varphi(r) \frac{x}{r}$ between $\lambda$ and $\Psi(r) \ dx$
to be a contraction
n is that
$$
 (r  \Psi^{\frac{1}{d-1}}(r))' \ge  1.
$$
\end{example}

\begin{corollary}
In $d$-dimensional Euclidean space with 
density $\Psi(r)$ satisfying $ (r  \Psi^{\frac{1}{d-1}}(r))' \ge  1$, 
the Euclidean isoperimetric inequality holds.
\end{corollary}

Some example of contraction mappings arise naturally  in differential geometry 
(see \cite{MaurMorg}, Propositions 1.1 and  2.1).

\begin{proposition}
Let $M$ be the plane equipped with the  metric $$dr^2 + g^2(r) r^2 d\theta^2$$ (surface of revolution), $g \ge 1$. Then
the identity mapping form $M$ to the Euqlidean plane with measure $g \ dx$
is a volume preserving contraction.

In particular, $\cosh^2(r) \ dx$
is a  Lipschitz image of $H^2$ (with metric $dr^2 + \cosh^2(r) d \theta^2$).
\end{proposition}

The following comparison result has been proved in \cite{KoZh}. It turns out that a natural model measure
for the one-dimensional log-convex distributions  has the following form:
$$
\nu_{A} = \frac{dx}{\cos Ax}, \ -\frac{\pi}{2A}  < x < \frac{\pi}{2A}.
$$
Its potential $V$  satisfies
$
V'' e^{-2V} =A^2.
$
Using a result \cite{RCBM} on symmetricity of the isoperimetric sets one can compute
the isoperimetric profile of $\nu_A$:
$$
\mathcal{I}_{\nu_A}(t) =  e^{At/2}+e^{-At/2}.
$$

\begin{proposition}
Let $\mu = e^{W} dx$ be a measure on $\R^1$ with even convex potential $W$.
Assume that
$$W''e^{-2W} \ge A^2,$$
and $W(0)=0$. Then $\mu$ is the image of $\nu_A$ under
a $1$-Lipschitz increasing mapping.
\end{proposition}
\begin{proof}
Without loss of generality one can assume that $W$ is smooth and $W''e^{-2W} > A^2$.
Let $\varphi$ be a convex potential such that
$T=\varphi'$ sends $\mu$ to $\nu_A$.
In addition, we require that $T$ is antisymmetric.
Clearly, $\varphi'$ satisfies
$$
e^{W}  = \frac{\varphi''}{\cos A \varphi'}.
$$
Assume that $x_0$ is a local maximum point for $\varphi''$.
Then at this point
$$
\varphi^{(3)}(x_0)=0, \ \ \varphi^{(4)}(x_0) \ge 0.
$$
Differentiating the change of variables formula at $x_0$ twice we get
$$
W''
= \frac{\varphi^{(4)}}{\varphi''}
- \Bigl( \frac{\varphi^{(3)}}{\varphi''}\Bigr)^2 + \frac{A^2}{\cos^2 A \varphi'}
(\varphi'')^2 + A \frac{\sin A \varphi'}{\cos A \varphi'} \varphi^{'''}.
$$
Consequently one has at $x_0$
$$
W''
\le \frac{A^2}{\cos^2 A \varphi'} (\varphi'')^2 = A^2 e^{2W}.
$$
But this contradicts to the main assumption.

Hence $\varphi''$ has no local maximum. Note that $\varphi$ is even. This implies that that $0$
is the global minimum of $\varphi''$. Hence
$
\varphi'' \ge \varphi''(0)=1.
$
Clearly, $T^{-1}$ is the desired mapping.
\end{proof}

\section{Other results and applications}

An immediate consequence of the contraction theorem is the 
Bakry-Ledoux comparison theorem, which is a probabilistic analog of the L{\'e}vy-Gromov comparison theorem for Ricci positive manifolds. 

\begin{theorem}
Assume that $\mu =e^{-V} dx$, where $D^2 V \ge \mbox{\rm{Id}}$, is a probability measure on $\R^d$.
Then 
$$
\mathcal{I}_{\mu} \ge \mathcal{I}_{\gamma},
$$
where $\gamma$ is the standard  Gaussian measure.
\end{theorem}

In the same way the contraction theorem implies different functional and concentration inequalities for uniformly log-concave measures 
(log-Sobolev, Poincar{\'e} etc.). 

The following unsolved problem is known as the ''Gaussian correlation conjecture''.

{\bf Gaussian correlation conjecture.}
 Let $A$ and $B$ be symmetric convex sets and $\gamma$ be the standard Gaussian measure.
Then
\begin{equation}
 \label{gcc}
\gamma(A \cap B) \ge  \gamma(A) \gamma(B).
\end{equation}

The Gaussian correlation conjecture has quite a long history. This problem arose in 70th.  
 The positive solution is known for two-dimensional sets and for the case when one of the sets is an ellipsoid.
The ellipsoid case  was proved by G. Harg{\'e} \cite{Harge} by semigroup arguments.

\begin{theorem}
\label{ellipsoid}
 Let $B$ be an ellipsoid. Then (\ref{gcc}) holds.
\end{theorem}
\begin{proof}
 Applying a linear transfromation of measures one can reduce the proof to the case when
$B$ is a ball and $\gamma$ is a (non-standard) Gaussian measure. 
Consider the optimal transportation $T$ between $\gamma$ and $\gamma_A = \frac{1}{\gamma(A)} \gamma_{|A}$.
By Theorem \ref{contr2}
$T$ is a contraction and by the symmetry reasons $T(0)=0$. Hence $T(B) \subset B$ and
$$
\frac{\gamma(A \cap B)}{\gamma(A)} = \gamma_A(B) = \gamma(T^{-1}(B)) \ge 
\gamma(B).
$$
This completes the proof.
\end{proof}

The following beautiful  observation \cite{Harge2} follows from the contraction theorem and properties of  the Ornstein-Uhlenbeck semigroup
$$
P_t f(x) = \int f(x  e^{-t} + \sqrt{1-e^{-2t}} y) \ d \gamma(y).
$$

\begin{theorem}
 If $\gamma$ is a standard Gaussian measure, $g$ is symmetric convex and $f $ is symmetric log-concave, then
$$
\int f g \ d \gamma \le \int f  d \gamma \cdot \int g \ d \gamma. 
$$
\end{theorem}
\begin{proof}
Let $T(x) =x + \nabla \varphi(x)$ be the optimal transportation of $\gamma$ onto $\frac{f \cdot \gamma}{\int f  d\gamma}$.
Thus we need to prove that
$$
\int g(x + \nabla \varphi(x)) \ d \gamma   \le \int g  \ d \gamma.
$$
Set:
$$
\psi(t) = \int g(x + P_t (\nabla \varphi(x)) ) \ d \gamma,
$$
where $P_t = e^{tL}$ is the Ornstein-Uhlenbeck semigroup generated by $L = \Delta - \langle x, \nabla \rangle$.
Note that
$$
\frac{\partial}{\partial t} \psi(t) 
=
\frac{\partial}{\partial t} \int g(x + P_t (\nabla \varphi(x))) \ d \gamma
=
\int \langle  \nabla g(x + P_t (\nabla \varphi(x))),  L P_t (\nabla \varphi(x))  \rangle  \  d \gamma.
$$
Integrating by parts we get
$$
\frac{\partial}{\partial t} \psi(t) =
- \int \mbox{Tr} \Bigl[ D^2 g (x + P_t (\nabla \varphi(x))) \cdot (I + M) M\Bigr]  \ d \gamma,
$$
where
$$
M = D P_t (\nabla \varphi(x)) = e^{-t/2} P_t ( D^2 \varphi).
$$
Clearly, by the contraction theorem $I + M \ge 0$ and $M \le 0$. Hence $\mbox{Tr} \Bigl[ D^2 g \cdot (I + M) M \Bigr] \le 0$ and $\psi(t)$ is increasing.
Note that $P_{+\infty}(\nabla \varphi) = \int  \nabla \varphi \ d\gamma =  \frac{\int x f \ d \gamma}{\int f \ d \gamma} = 0$.
 Hence
$ \int g(x + \nabla \varphi(x)) \ d \gamma \le \psi(+ \infty) = \int g \ d \gamma $. The proof is complete.
\end{proof}

Some other  applications to correlation inequalities have been obtained in 
\cite{CorEr}, \cite{KimMilman}.
A generalization of Theorem \ref{ellipsoid} to non-Gaussian measures have been 
obtained in \cite{KimMilman}
(see Corollary 4.1).

Other applications obtained in \cite{Caf}, \cite{CorEr}, \cite{Harge}, \cite{KimMilman}
concern inequalities
of the type
$$
\int \Gamma(x) \ d\mu \le \int \Gamma(x) \ d \nu,
$$
where $\Gamma(x)$ is  convex (moment inequalities etc.).

The following theorem was obtained in \cite{CEFM} with the help of the contraction theorem. In particular, it solves the so-called (B)-conjecture 
from the theory of Gaussian measures.

\begin{theorem}
Let $K$ be a symmetric convex set and $\gamma$ is a standard Gaussian measure. Then the function
$$
t \to \gamma(e^t K)
$$
is log-concave.

In particular, $\gamma(\sqrt{ab} K)^2 \ge \gamma(aK) \gamma(bK)$ for every $a>0, b>0$.
\end{theorem}
{\bf Sketch of the proof.} 
Since $\gamma(e^{t_1 + t_2} K) = \gamma(e^{t_1} (e^{t_2} K))$, it is sufficient to show that $g(t) = \gamma(e^t K)$
is log-concave at zero. This is equivalent to the inequality $g''(0) g(0) \le (g'(0))^2$. Computing the derivatives of $g$
we get that this is equivalent to 
$$
\int |x|^4 \ d \gamma_K - \bigl( \int |x|^2 \ d \gamma_K\bigr)^2 \le 2 \int x^2 \ d \gamma_K,
$$  
where $\gamma_K = \frac{1}{\gamma(K)} I_{K} \cdot \gamma$.
Let us prove a more general relation: if $\mu = e^{-W} \ dx$ is a log-concave measure with $D^2 W \ge \mbox{Id}$ and $f$
is a function, satisfying $\int f \ d\mu = 0, \int \nabla f \ d \mu=0$, then the following Poincar{\'e}-type inequality holds:
\begin{equation}
\label{strongPoin}
\int f^2 \ d \mu \le \frac{1}{2} \int \| \nabla f \|^2 \ d \mu.
\end{equation}
Applying (\ref{strongPoin}) to $f = |x|^2 - \int |x|^2 \ d \mu$, we get the desired inequality for $\mu$. Then it remains
to approximate $\gamma_K$ by measures of this type.

Note that by the Caffarelli's theorem it is sufficient to prove inequality (\ref{strongPoin}) only for the standard Gaussian measure.
But in this case  (\ref{strongPoin}) is well-known and can be obtained from the expansion of $f$ on the basis formed by the Hermite polynomials. 
The proof is complete.

\

Note that apart from the observations of the previous section nothing is known about contractions of manifolds.

The following result was obtained by S. I. Valdimarsson (see \cite{Vald}).
For every nonnegative symmetric $M$ let us denote by $\gamma_M$ the
Gaussian measure with density
$$
\sqrt{\det M} e^{-\pi \langle M x, x \rangle}.
$$

\begin{theorem}
Let $A$, $G$ and $B$ are positive definite symmetric linear transformations, $A<G$, $GB=BG$, $H$ is  a convex function,
and $\mu_0$ is a probability measure.
The optimal transportation $T = \nabla \Phi$  between probability measures
$$
\mu = \gamma_{B^{-1/2} G B^{-1/2}} \ast \mu_0 \ \mbox{and} \ \nu = C e^{-H}  \cdot\gamma_{B^{-1/2} A^{-1}  B^{-1/2}} 
$$
satisfies
$$
D^2 \Phi \le G.
$$
\end{theorem}

A particular form of the measure $\mu$ allows Valdimarsson (after F. Barthe \cite{Barthe}) to obtain by the transportation arguments a 
new form of the well-known  Brascamp-Lieb  inequality. See \cite{Vald} for details.

We finish with the following observation from \cite{BarKol}.

\begin{proposition}
Let $\mu = I_{[0,+\infty)} e^{-x} \ dx $ be the one-sided exponential measure and $\nu = e^{g} \cdot \mu$ with $|g'| \le c$ for some $c<1$. Then the monotone map $T$ which transports $\nu$ to $\mu$ satisfies
$$T'(x) \in [1-c, 1+c]$$
for all $x \in [0, \infty)$. The inverse map $S = T^{-1}$ is a $\frac{1}{1-c}$-contraction.
\end{proposition}

The result follows from the explicit representation of $T$ but can heuristically proved by the maximum principle arguments applied to $S$.
Indeed
$$
g(S) -S + \log S' = - x.
$$
If $x_0$ is the maximum point for $S'$, one has $S''(x_0) = 0$. In addition,
$$
g'(S(x_0)) S'(x_0)-S'(x_0) + \frac{S''(x_0)}{S'(x_0)} = -1.
$$
Clearly $S'(x_0) = \frac{1}{1-g'(S(x_0)} \le \frac{1}{1-c}$.

Using this property one can give a transportation proof of the 
$1$-dimensional Talagrand inequality for the  exponential law (see \cite{BarKol}, Proposition 6.6).

\end{document}